\documentclass[12pt,a4paper,makeidx]{amsart}
\usepackage[latin1]{inputenc}        % accents
\usepackage[dvips]{graphics}
\usepackage[dvips]{graphicx}
\usepackage{mathrsfs}
\let\mathcal\mathscr
\usepackage{amssymb}
\usepackage{amsmath}
\usepackage{amsfonts}
\usepackage{latexsym}
\usepackage[T1]{fontenc}
\usepackage[latin1]{inputenc}
\usepackage{hyperref}
\usepackage[all,ps]{xy}
\RequirePackage{amsthm}
\RequirePackage{amssymb}
\usepackage[all,ps]{xy}
\usepackage{latexsym}
\usepackage{amscd}
\usepackage[latin1]{inputenc}
\usepackage[T1]{fontenc}
\usepackage{color}
\RequirePackage{amsthm}
\RequirePackage{amssymb}
\usepackage[all,ps]{xy}
\usepackage{latexsym}

\def\B{{\mathbb B}}

\def\R{{\mathbb R}}
\def\Q{{\mathbb Q}}
\def\C{{\mathbb C}}

\def\supp{\mathop{\rm Supp}\nolimits}

\def\phi{\varphi}
\def\a{{\alpha}}

\def\e{{\varepsilon}}

\def\D{{\Delta}}
\def\G{{\Gamma}}

\def\cI{{\mathcal I}}

\def\cO{{\mathcal O}}

\def\nn{\mathop{\rm NNef}\nolimits}
\def\vn{\mathop{v_{\rm num}}\nolimits}

\def\exc{\mathop{\rm Exc}\nolimits}

\def\div{\mathop{\rm div}\nolimits}

\def\le{\leqslant}
\def\ge{\geqslant}

\newtheorem{thm}{Theorem}[section]
\newtheorem{defn}[thm]{Definition}

\newtheorem{prop}[thm]{Proposition}
\newtheorem{lem}[thm]{Lemma}

\newtheorem{con}[thm]{Conjecture}

\newtheorem*{thmA}{Theorem A} 
\newtheorem*{thmB}{Theorem B}

\title[Uniruledness of  stable base loci of adjoint linear systems]{Uniruledness of  stable base loci of adjoint linear systems with and without Mori Theory}
\author[S. Boucksom, A. Broustet and G.~Pacienza]{S\'ebastien Boucksom, Ama\"el Broustet and Gianluca Pacienza}
\date{\today}

\begin{document}
\maketitle
{\let\thefootnote\relax
\footnote{\hskip-1.2em
\textbf{Key-words :} big and pseff adjoints divisors; stable, non-ample and non-nef base locus; rational curves.\\
\noindent
\textbf{A.M.S.~classification :} 14J40. 
 }}
\numberwithin{equation}{section}

%%%%%%%%%%%%%%%%%%%%%%%%%
%
\begin{abstract} We explain how to deduce from recent results in the Minimal Model Program a general uniruledness theorem for base loci of adjoint divisors. We also show how to recover special cases by extending a technique introduced by Takayama. \end{abstract}
%
%%%%%%%%%%%%%%%%%%%%%%%%%%
%
%%%%%%%%%%%%%%%%%%%%%%%%
%
\section*{Introduction}
%
%%%%%%%%%%%%%%%%%%%%%%%%

Let $X$ be a normal projective variety defined over $\C$ (or any algebraically closed field of characteristic $0$) and let $D$ be an $\R$-divisor on $X$ (where $\R$-divisor will mean $\R$-Cartier $\R$-divisor unless otherwise specified). Following~\cite{BCHM} one introduces the (real) \emph{stable base locus} of $D$ as
\begin{equation}\label{equ:stable}\B(D):=\bigcap\{\supp E,\,E\text{ effective }\R\text{-divisor},\,E\sim_\R D\},
\end{equation}
where $E\sim_\R D$ means that $E$ is $\R$-linearly equivalent to $D$, i.e. $E-D$ is an $\R$-linear combination of principal divisors $\div(f),\,f\in\C(X)$. When $D$ is a $\Q$-divisor $\B(D)$ coincides with the usual stable base locus (cf.~Proposition~\ref{prop:stable} below). 

As in~\cite{ELMNP1} one then defines the \emph{augmented base locus} of $D$ by\begin{equation}\label{equ:augmented}
\B_+(D):=\bigcap_{m>0}\B(D-\frac{1}{m}A)
\end{equation}
and the \emph{restricted base locus} of $D$ by
\begin{equation}\label{equ:restricted}
\B_-(D):=\bigcup_{m>0}\B(D+\frac{1}{m}A)
\end{equation}
where $A$ is an ample divisor, the definition being independent of $A$. We thus have the inclusions
$$
\B_-(D)\subset\B(D)\subset\B_+(D).
$$

The augmented base locus $\B_+(D)$ is Zariski closed and satisfies 
$$
\B_+(D)\subsetneq X\Longleftrightarrow D\text{ big},
$$ 
$$
\B_+(D)=\emptyset\Longleftrightarrow D\text{ ample}.
$$ 
Augmented base loci are also known as \emph{non-ample loci} and have been extensively studied in relation with the asymptotic behavior of linear series (see~\cite{Nak,ELMNP1, ELMNP2} and~\cite{B} for the analytic counterpart). 

The restricted base locus $\B_-(D)$ is an at most  countable union of Zariski closed sets - it might not be Zariski closed in general even though no specific example seems to be known for the moment. We have 
$$
\B_-(D)\subsetneq X\Longleftrightarrow D\text{ pseudoeffective},
$$ 
$$
\B_-(D)=\emptyset\Longleftrightarrow D\text{ nef}.
$$ 
On the other hand the \emph{non-nef locus} $\nn(D)$ of an $\R$-divisor $D$~\cite{B,BDPP}, called the \emph{numerical base locus} in~\cite{Naka}, is defined in terms of the asymptotic or numerical vanishing orders attached to $D$ (cf. Definition~\ref{defn:nonnef} below). We always have
$$
\nn(D)\subset\B_-(D)
$$
and equality was shown to hold when $X$ is smooth in~\cite{ELMNP1}, but seems to be unknown when $X$ is an arbitrary normal variety.

The goal of the present paper is to investigate the uniruledness properties of the above loci in the case of adjoint divisors. After having collected basic facts in Section 1, we explain in Section 2 how to obtain the following general result using known parts of the Minimal Model Program~\cite{Kaw,BCHM}. 

\begin{thmA} Let $X$ be a normal projective variety and let $\D$ be an effective $\R$-Weil divisor such that $(X,\D)$ is klt. 
\begin{enumerate} 
\item[(i)] We have $\nn(K_X+\D)=\B_-(K_X+\D)$ and each of its irreducible components is uniruled. 
\item[(ii)] If $K_X+\D$ is furthermore big then 
$$
\nn(D)=\B_-(K_X+\D)=\B(K_X+\D)
$$ 
and every irreducible component of $\B_+(K_X+\D)$ is uniruled as well. 
\end{enumerate}
\end{thmA}
As already noticed in~\cite{taka2} the above uniruledness results both fail in general for the more general case of where $(X,\D)$ has log-canonical singularities, even in the log-smooth case (cf. Example 6.4).

The special case of Theorem A where $X$ is smooth and either $K_X$ or $\D$ vanishes was obtained by S.~Takayama in~\cite{taka2} by a completely different (and more direct) method, which combined his extension result for log-pluricanonical forms (see \cite[Theorem 4.5]{taka}) with the characterization of uniruled varieties in terms of the non-pseudo-effectivity of the canonical class already mentioned. 

In Section 3, we show more generally how to obtain using Takayama's method the following special cases of Theorem A. 

 \begin{thmB} Let $X$ be a smooth projective variety and $L$ a line bundle on $X$ such that either $-K_X$ of $L-K_X$ is nef. 
 \begin{enumerate} 
\item[(i)] If $L$ is pseudoeffective, then every irreducible component  of the non-nef locus $\B_-(L)$ is uniruled.
\item[(ii)] If $L$ is furthermore big, then every irreducible component  of the stable base locus $\B(L)$ or  of the non-ample locus $\B_+(L)$ is uniruled.
\end{enumerate}
\end{thmB}

{\bf Acknowledgements.} The authors would like to thank St\'ephane Druel for interesting exchanges related to this work. Part of this work was done by G.P. during his stay at the Universit\`a di Roma "La Sapienza". He wishes to thank Kieran O'Grady for making this stay very pleasant and stimulating  and for providing the financial support. A.B. thanks Laurent Bonavero for  stimulating conversations on this subject.

%%%%%%%%%%%%%%%%%%%%%%%%
%
\section{Preliminaries}\label{S:prel}
Unless otherwise specified we will use the standard notation, definitions and terminology (cf. for instance \cite{KM}). By convention divisor (resp. $\Q$-divisor, $\R$-divisor) will mean Cartier divisor (resp. $\Q$-Cartier, $\R$-Cartier) unless otherwise specified. 

\subsection{Approximation by $\Q$-divisors}
Let $X$ be a normal projective variety. Recall that the stable base locus of a $\Q$-divisor, that we temporarily denote by $\B_\Q(D)$, can be described as follows:
$$
\B_\Q(D):=\bigcap\{\supp E,\,\,E\text{ effective }\Q\text{-divisor},\,E\sim_\Q D\}.
$$
\begin{prop}\label{prop:stable} Let $D$ be a $\Q$-divisor on $X$. Then its real stable locus $\B(D)$ defined by (\ref{equ:stable}) coincides with the usual stable locus $\B_\Q(D)$. 
\end{prop}
\begin{proof} It is obvious that $\B(D)\subset\B_\Q(D)$. Conversely let $E$ be an effective $\R$-divisor such that $E\sim_\R D$. By Lemma~\ref{lem:approx} below we may find an effective $\Q$-divisor $E'\sim_\R D$ with the same support as $E$ and the result follows.
\end{proof}

\begin{lem}\label{lem:approx} Let $D$ be a $\Q$-Cartier divisor and let $E$ be an effective $\R$-Cartier divisor such that $E\sim_\R D$. Then $E$ may be written as a (coefficient-wise) limit of effective $\Q$-Cartier divisors $E_j$ with the same support as $E$ and such that $E_j\sim_\R D$. 
\end{lem}
\begin{proof} Denote by $W_\R(X)\supset C_\R(X)$ the space of $\R$-Weil divisors and the subspace of $\R$-Cartier divisors respectively. Let $V$ be the finite dimensional $\R$-vector subspace of $W_\R(X)$ spanned by the irreducible components of $E$. Then $V$ is defined over $\Q$, and so is the affine space of all $\R$-Cartier divisors linearly equivalent to $D$ since the latter is a $\Q$-divisor. As a consequence 
$$
W:=V\cap\{F\in C_\R(X),\,F\sim_\R D\}
$$ 
is an affine subspace of $V$ defined over $\Q$. Since $W$ contains $E$, the latter may then approximated inside $V$ by elements of $W\cap V(\Q)$, which yields the result.
\end{proof}

\subsection{Augmented base loci}
We collect in this section some preliminary results regarding augmented base loci.  We shall use the following common terminology.
\begin{defn}[Kodaira decompositions] Let $X$ be a normal projective variety and $D$ be a big $\R$-divisor on $X$. A \emph{Kodaira decomposition} of $D$ is a decomposition $D=A+E$ into $\R$-divisors with $A$ ample and $E$ effective.
\end{defn}
By (\cite{ELMNP1}, Remark 1.3) the augmented base locus of a big $\R$-divisor $D$ can be described as
\begin{equation}\label{equ:bplus}
\B_+(D):=\bigcap_{D=A+E}\supp E,
\end{equation}
where the intersection runs over all Kodaira decompositions of $D$. The following result shows that one obtains the same locus by allowing Kodaira decompositions on birational models. 

\begin{lem}\label{lem:birb} Let $X$ be a normal projective variety and let $D$ be a big $\R$-divisor on $X$. Then its augmented base locus satisfies
$$
\B_+(D)=\bigcap_{\pi^*D=A+E}\pi(\supp E)
$$
where $\pi$ ranges over all birational morphisms $X'\to X$ and $\pi^*D=A+E$ over all Kodaira decompositions of $\pi^*D$ on $X'$. 
\end{lem} 
\begin{proof} In view of (\ref{equ:bplus}) it is clear that 
$$
\B_+(D)\subset\bigcap_{\pi^*D=A+E}\pi(\supp E). 
$$
Consider conversely a birational morphism $\pi:X'\to X$ and a Kodaira decomposition 
$$\pi^*D=A+E$$ 
on $X'$ and let $x\in X-\pi(\supp E)$. We have to show that $x\in X-\B_+(D)$. Since $E=\pi^*D-A$ is both effective and $\pi$-antiample, its support must contain every curve contracted by $\pi$, i.e. the exceptional locus $\exc(\pi)$ is contained in $\supp E$. Since $x\notin\pi(\supp E)$ it follows that there is a unique preimage $x'$ of $x$ by $\pi$ and that $x'\notin\supp E$. Now let $B$ be a small enough ample divisor on $X$, so that $A-\pi^*B$ is ample on $X'$. We then have $\B(A-\pi^*B)=\emptyset$, which means that there exists an effective $\R$-divisor $F$ on $X'$ with 
$$
F\sim_\R A-\pi^*B
$$ 
and such that $x'\notin\supp F$. As a consequence $x'$ doesn't belong to the support of the effective $\R$-divisor $G':=E+F$. Since $G'$ is $\R$-linearly equivalent to $\pi^*(D-B)$ there exists an effective $\R$-divisor $G\sim_\R D-B$ on $X$ such that $\pi^*G=G'$, and $x'\notin\supp G'$ implies $x=\pi(x')\notin\supp G$. We have thus constructed a Kodaira decomposition $D=B+G$ with $x\notin\supp G$, which shows that $x\notin\B_+(D)$ as desired. 
\end{proof}

The next result describes the behavior of augmented base loci under birational transforms.  
\begin{prop}\label{prop:biratB+}
Let $\pi : X\rightarrow Y$ a birational morphism between normal projective varieties. Then for any big $\R$-divisor $D$ on $Y$ and any effective $\pi$-exceptional $\R$-divisor $F$ on $X$ we have
$$
\B_+(\pi^*D+F)=\pi^{-1}(\B_+(D))\cup\exc(\pi).$$
\end{prop}
\begin{proof} Let $x\in X-\B_+(\pi^*D+F)$, so that there exists a Kodaira decomposition 
$$\pi^*D+F=A+E$$
with $x\notin\supp E$. Then $G:=E-F$ is $\pi$-antiample and $\pi_*G=\pi_*E$ is effective since $F$ is $\pi$-exceptional, thus the so-called "negativity lemma" (\cite{KM}, Lemma 3.39) shows that $G$ is effective. Since it is also $\pi$-antiample it must contain $\exc(\pi)$ in its support. We thus get a Kodaira decomposition 
$$
\pi^*D=A+G
$$ 
such that $\pi(x)\notin\pi(\supp G)$, and Lemma \ref{lem:birb} implies that $\pi(x)\notin\B_+(D)$. This shows that 
$$
\pi^{-1}(\B_+(D))\cup\exc(\pi)\subset\B_+(\pi^*D+F).
$$

In order to prove the reverse inclusion we first consider the special case where $D=A$ is an ample $\Q$-Cartier divisor on $Y$ and $F=0$. Our goal is then to show that $\B_+(\pi^*A)\subset\exc(\pi)$. Pick $x\notin\exc(\pi)$ and choose a hyperplane section $H$ of $X$ such that $x\notin H$. Since $\pi$ is an isomorphism above  $\pi(x)$ it follows that $\pi(x)$ doesn't belong to the zero locus of the ideal sheaf $\cI:=\pi_*\cO_X(-H)$. If we choose $k$ sufficiently large and divisible then $\cO_Y(kA)\otimes\cI$ is globally generated since $A$ is an ample $\Q$-divisor and we get the existence of a section in $H^0(Y,\cO_Y(kA)\otimes\cI)$ that doesn't vanish at $\pi(x)$, hence a section $s\in H^0(X,k\pi^*A-H)$ with $s(x)\neq 0$, which indeed shows that $x\notin\B_+(\pi^*A)$. 

We now treat the general case. We thus pick $x\in X-\exc(\pi)$ such that $\pi(x)\notin\B_+(D)$, and we have to show that $x\notin\B_+(\pi^*D+F)$. Since $\pi(x)\notin\B_+(D)$ there exists a Kodaira decomposition
$$
D=A+E
$$ 
with $\pi(x)\notin\supp E$, and we may assume that $A$ is $\Q$-Cartier by Lemma \ref{lem:approx}. By the special case treated above we have $\B_+(\pi^*A)\subset\exc(\pi)$, so that there exists a Kodaira decomposition 
$$
\pi^*A=B+G
$$ with $B$ ample and $x\notin\supp G$. Putting all together yields a Kodaira decomposition
$$
\pi^*D+F=B+(G+E+F)
$$
with $x\notin\supp(G+E+F)$, which concludes the proof.
\end{proof}

\subsection{Restricted base loci vs. non-nef loci}\label{sec:nonnef}
Let $D$ be a big $\R$-divisor on the normal projective variety $X$. Given a divisorial valuation $v$ on $X$ we may define the \emph{numerical vanishing order} of $D$ along $v$ by
$$
\vn(D):=\inf\{v(E),\,E\text{ effective }\R\text{-divisor},\,E\equiv D\},
$$
where $\equiv$ denotes numerical equivalence. It also satisfies 
\begin{equation}\label{equ:vn}
\vn(D)=\inf\{v(E),\,E\text{ effective }\R\text{-divisor},\,E\sim_\R D\}
\end{equation}
by~\cite{ELMNP1} Lemma 3.3. The induced function on the open convex cone 
$$
\text{Big}(X)\subset N^1(X)
$$
of big classes is homogeneous and convex, hence continuous and sub-additive. When $D$ is a pseudoeffective $\R$-divisor we set following~\cite{Naka,B} 
\begin{equation}\label{equ:num}
\vn(D):=\lim_{\e\to 0}\vn(D+\e A)
\end{equation}
with $A$ ample. This is easily seen to be independent of the choice of $A$. As shown in~\cite{Naka,B} the corresponding function on the pseudoeffective cone 
$$
\text{Psef}(X)=\overline{\text{Big}(X)}\subset N^1(X).
$$ 
is lower semicontinuous, but \emph{not} continuous up to the boundary of the pseudoeffective cone in general(cf. \cite{Naka} Example 2.8 p.135) and a pseudoeffective $\R$-divisor $D$ is nef iff $\vn(D)=0$ for every divisorial valuation $v$. 

 \begin{lem}\label{lem:bir} Let $\pi:X'\to X$ be a birational morphism and let $D$ be a pseudoeffective $\R$-divisor on $X$. Then we have 
$$\vn(\pi^*D)=\vn(D)$$
for every divisorial valuation $v$.
\end{lem} 
\begin{proof}
This is clear when $D$ is big by (\ref{equ:vn}). Let now $D$ be pseudoeffective and pick an ample divisor $A$ on $X$. For every $\e>0$ $D+\e A$ is big thus we have
$$
\vn(D+\e A)=\vn(\pi^*D+\e\pi^*A)\le\vn(\pi^*D)
$$
by subadditivity of $\vn$ since $\vn(\e\pi^*A)=0$, $\pi^*A$ being nef. On the other the lower semicontinuity of $\vn$ on $\text{Psef}(X')$ implies that
$$
\vn(\pi^*D)\le\liminf_{\e\to 0}\vn(\pi^*D+\e\pi^*A)
$$
and the result follows. 
\end{proof}

\begin{defn}\label{defn:nonnef} Let $D$ be an $\R$-divisor on $X$. The \emph{non-nef locus}~\cite{B} (or \emph{numerical base locus}~\cite{Naka}) of $D$ is defined by
$$\nn(D):=\bigcup\{c_X(v),\,\vn(D)>0\},$$
where $c_X(v)$ denotes the center on $X$ of a given divisorial valuation $v$, when $D$ is pseudoeffective. When $D$ is not pseudoeffective one sets $\nn(D)=X$. 
\end{defn}
The non-nef locus is always contained in the restricted base locus:

\begin{lem} For every $\R$-divisor $D$ we have 
$$
\nn(D)\subset\B_-(D).
$$
\end{lem}
\begin{proof} If $D$ is not pseudoeffective then $\B_-(D)=X$ by \cite{ELMNP1}. We may thus assume that $D$ is pseudoeffective. Let $x\notin\B_-(D)$. Given an ample divisor $A$ we have $x\notin\B(D+\e A)$ for each $\e>0$, thus there exists an effective $\R$-divisor $E_\e\sim_\R D+\e A$ such that $x\notin\supp E_\e$, and we infer that 
$$
\vn(D+\e A)\le\vn(E_\e)=0
$$
for each divisorial valuation $v$ such that $x\in c_X(v)$. Letting $\e\to 0$ yields $\vn(D)=0$ for such divisorial valuations, and we conclude that $x\notin\nn(D)$ as desired. 
\end{proof}

When $X$ is \emph{smooth} it was shown in~\cite{ELMNP1} Proposition 2.8, using Nadel's vanishing theorem, that equality holds, i.e.  
$$
\nn(D)=\B_-(D)
$$
for every pseudoeffective $\R$-divisor $D$. This  shows in particular that $\nn(D)$ is an at most countable union of Zariski closed subsets of $X$. This property holds as well when $X$ is an arbitrary normal variety since choosing a resolution of singularities $\pi:X'\to X$ yields
$$\nn(D)=\pi(\nn(\pi^*D))$$
by Lemma~\ref{lem:bir}. 

On the other hand one may wonder whether the equality $\nn=\B_-$ holds more generally on all normal projective varieties $X$. This is easily seen to be equivalent to the following:
\begin{con} Let $L$ be a big line bundle on a normal projective variety $X$. Let $x\in X$ be such that for each divisorial valuation $\nu$ centered at $x$ there exists an infinite sequence $\sigma_k\in H^0(kL)$ such that $\nu(\sigma_k)=o(k)$. Then there exists an ample divisor $A$ and an infinite sequence $\tau_k\in H^0(kL+A)$ such that $\tau_k(x)\neq 0$. 
\end{con}

Using~\cite{BCHM} we prove:
\begin{prop}\label{prop:baseloci} Let $(X,\D)$ be a klt pair. Then we have $\nn(K_X+\D)=\B_-(K_X+\D)$, which furthermore coincides with $\B(K_X+\D)$ when $K_X+\D$ is big.
\end{prop}
\begin{proof} We may assume that $K_X+\D$ is pseudoeffective, since the result is clear otherwise. Given an irreducible component $V$ of $\B_-(D)$ there exists an ample $\R$-divisor $A$ such that $V$ is a component of $\B(D+2A)\subset\B_-(D+A)$, and upon changing $A$ in its $\R$-linear equivalence class we may assume that $(X,\D+A)$ is klt. We thus see that we may assume that $K_X+\D$ is big to begin with, and it is then enough to show that $\nn(K_X+\D)=\B(K_X+\D)$ since the latter contains $\B_-(K_X+\D)$. 

By~\cite{BCHM} $K_X+\D$ admits an ample model, which means that there exist birational morphisms $\pi:Y\to X$ and $\pi':Y\to X'$ such that 
$$\pi^*(K_X+\D)=\pi'^*H+F$$ 
where $H$ is ample on $X'$ and $E$ is effective and $\pi'$-exceptional, and $Y$ may be assumed to be smooth. By the "negativity lemma" (\cite{KM}, Lemma 3.39) every effective $\R$-divisor $E$ on $Y$ such that $E\equiv\pi'^*H+F$ satisfies $E\ge F$, and it easily follows that
$$\nu_{\text{num}}(K_X+\D)=\nu_{\text{num}}(\pi^*(K_X+\D))=\nu(F)$$
for every divisorial valuation $\nu$, so that 
$$
\nn(K_X+\D)=\pi(\supp F).
$$
On the other hand we have 
$$
\B(K_X+\D)=\pi(\B(\pi^*(K_X+\D))=\pi(\supp F)
$$
and the result follows. 
\end{proof}

\section{Proof of Theorem A}
Let $X$ be a normal projective and let $\D$ be an effective $\R$-Weil
divisor on $X$ such that $(X,\D)$ is klt.
If $K_X+\D$ is not pseudoeffective, then by \cite{BDPP} $X$ is
uniruled : in fact, considering a log-resolution $f : Y \rightarrow X$
of $(X,\D)$, and an effective divisor $\G$ such that
$$K_Y + \G = f^*(K_X +\D) +E,$$ 
with $E$ $f$-exceptional, we have that $K_Y + \G$ is not pseudoeffective,
since $E$ is $f$-exceptional and $f^*(K_X +\D)$ not
pseudoeffective. As $\G$ is effective, $K_Y$ is not
pseudoeffective either, thus $Y$ is uniruled and $X=\nn(K_X+\D)=\B_-(K_X+\D)$ too.
 
%If $K_X+\D$ is not pseudoeffective then by \cite{BCHM} we can run a (directed) MMP for $(X,\D)$ ending up in a Fano fibration, and it follows that $X=\nn(D)=\B_-(D)$ is uniruled by \cite{Kaw}.  

Now assume that $K_X+\D$ is pseudoeffective and let $V$ be an irreducible component of $\B_-(K_X+\D)$. By \cite{ELMNP1} we have
$$
\B_-(K_X+\D)=\bigcup\{\B_+(K_X+\D+A),\, A  \text{ ample}\}
$$
thus there exists an ample $\R$-divisor $A$ such that $V$ is a component of $\B_+(K_X+\D+A)$. Since $A$ is ample we may furthermore assume that $(X,\D+A)$ is klt. Together with Proposition \ref{prop:baseloci} this reduces us to the following situation: assume that $(X,\D)$ is klt, $K_X+\D$ is big and let $V$ be an irreducible component of $\B_+(K_X+\D)$. We are then to show that $V$ is uniruled. 

Consider a commutative diagram of birational maps
\begin{equation}\label{equ:diag}
\xymatrix{
 X 
\ar[dr]_{\pi} \ar@{-->}[rr]^{\psi}&
& 
X'\ar[dl]^{\pi'}\\
 & Z.
 & \\
}
\end{equation}
with $-(K_X+\D)$ $\pi$-ample, and either $\pi$ is a divisorial contraction and $\pi'$ is the identity, or $\pi$ is a small contraction and $\pi'$ is its flip. Since $-(K_X+\D)$ is $\pi$-ample we have $\exc(\pi)\subset\B_+(K_X+\D)$. If $V$ is contained in $\exc(\pi)$ it must therefore be one of its irreducible components, and it follows that $V$ is uniruled by \cite{Kaw}. Otherwise we may consider its strict transform $V'$ on $X'$, since $\psi$ is in both cases an isomorphism away from $\exc(\pi)$. If we denote by $\D'$ the strict transform of $\D$ on $X'$ then $(X',\D')$ is klt and $K_{X'}+\D'$ is big. We claim that $V'$ is a component of $\B_+(K_{X'}+\D')$. 

Indeed consider a resolution of the indeterminancies of $\psi$
\begin{equation}\label{eq:bigdiag}
\xymatrix{ 
 &Y\ar[dl]_{\mu} \ar[dr]^{\mu'}\\
 X   \ar@{-->}[rr]^{\psi}& & X'
}
\end{equation}
which may be chosen such that $\mu$ (resp. $\mu'$) is an isomorphism above the generic point of $V$ (resp. $V'$). We have 
$$
\mu^*(K_X+\D)=\mu'^*(K_{X'}+\D')+F,
$$
where $F$ is $\mu'$-exceptional and $-F$ is nef over $X'$ (since it is nef over $Z$), thus $F\ge 0$ by the Negativity Lemma. The claim now follows by Proposition \ref{prop:biratB+}. 

By \cite{BCHM} there exists a finite composition of maps $\psi$ as in (\ref{equ:diag}) such that $K_{X'}+\D'$ is nef at the final stage, and by what we have just shown either the strict transform of $V$ is contained at in $\exc(\pi)$ at some stage, in which case it is uniruled, or the strict transform $V'$ on the final $X'$ is a component of $\B_+(K_{X'}+\D')$. By the base point free theorem there exists a further birational morphism $\rho:X'\to W$ such that $K_{X'}+\D'=\rho^*A$ with $A$ ample on $W$, and Proposition \ref{prop:biratB+} shows that $\B_+(K_{X'}+\D')=\exc(\rho)$, so that $V'$ is a component of $\exc(\rho)$. We then conclude that $V'$ is uniruled as desired by a final application of \cite{Kaw}.

%%%%%%%%%%%%%%%%%%%%%%%%
%
\section{Proof of Theorem B}
%
%%%%%%%%%%%%%%%%%%%%%%%%
In this section we first explain how to infer Theorem B from Theorem A, and then give a direct proof following Takayama's approach and thus avoiding~\cite{BCHM}. 
\subsection{Theorem A implies Theorem B}
We are actually going to show that Theorem A implies Theorem B when $L$ is an $\R$-divisor. As in the proof of Theorem A, we then have the flexibility to assume that $L$ is big upon adding to it a small multiple of an ample divisor. 

Assume first that $-K_X$ is nef. We then have 
$$\e L=K_X+(\e L-K_X)$$
and $\e L-K_X$ is numerically equivalent to a klt divisor $\D$ for $\e>0$ small enough. Indeed we can write $L\equiv A+E$ where $A$ is ample and $E$ is effective, hence 
$$\e L-K_X\equiv\e E+\e A-K_X$$
where $\e A-K_X$ is ample and $\e E$ is klt for $\e$ small enough. Since both $\B_-(L)$ and $\B_+(L)$ are invariant under scaling $L$ we thus get the result by Theorem A applied to $(X,\D)$. 

Now assume instead that $L-K_X=:N$ is nef. We can then write
$$\frac{1}{1-\e}L=K_X+N+\frac{\e}{1-\e} L$$
and $N+\frac{e}{1-\e}L$ is numerically equivalent to a klt divisor $\D$ for $\e>0$ small enough just as before, and Theorem A again implies the desired result after scaling $L$. 

\subsection{A (more) direct proof of Theorem B}
Takayama's key idea is that the proof of his extension result  \cite[Theorem 4.5]{taka}
may be used in combination with \cite{MM} and \cite{BDPP} to obtain the following criteria for uniruledness.

\begin{thm}[Takayama, \cite{taka2}, Corollary 3.3]\label{thm:takauniruled}
Let $X$  be a smooth projective variety and $V\subset X$ be an irreducible subvariety.  Let $D$ be a line bundle on $X$. Assume there exists a decomposition $D\equiv A+E$, where 
 $A$ is an ample  $\Q$-divisor and 
 $E$ is an effective $\Q$-divisor which is a maximal lc center for the pair $(X,E)$.
\begin{enumerate}
\item[(a)]\label{item:uni1}
If $V$ is \emph{contained} in the stable base locus $\B(K_X+D)$, then $V$ is uniruled.
\item[(b)]\label{item:uni2}
If $K_X+D$ is big and $V$ is a \emph{component} of the non-ample locus $\B_+(K_X+D)$, then $V$ is uniruled.
\end{enumerate}
\end{thm}
Recall that a maximal log-canonical (lc for short) center of $(X,E)$ is a subvariety along which the generic log-canonical threshold of $E$ is equal to $1$ (cf.~\cite{Laz}) and which is maximal for that property. 

We now consider the situation of Theorem B. Let thus $L$ be a line bundle and assume that either $-K_X$ or $L-K_X$ is nef. We begin with $(ii)$ of Theorem B. We thus assume that $L$ is big and let $V$ be an irreducible component of either $\B(L)$ or $\B_+(L)$ that is not contained in $\B_-(L)$. We try to apply Theorem~\ref{thm:takauniruled}. The desired Kodaira-type decomposition will be obtained thanks to the following result. 

\begin{lem}[\cite{taka2}, Proposition 4.3]\label{lem:decomp}
Let $X$ be a smooth projective variety and $D$ a big $\Q$-divisor on $X$. 
Assume that $V\subset X$ is an irreducible component of either
$\B(D)$ or $\B_+(D)$.
Then there exists a rational number $\alpha>0$ and a decomposition $\alpha D\equiv A+E$
with $A$ ample $\Q$-divisor and $E$ effective $\Q$-divisor on $X$ such that  $V$ is a maximal lc center for $(X,E)$.
\end{lem}

We can now extend \cite[Propositions 5.1 and 5.2]{taka2} as follows. 
\begin{prop}\label{prop1}
Let $X$ be a smooth projective variety and let $L$ be a big line bundle on $X$. Let $V$ be an irreducible component of either $\B(L)$ or $\B_+(L)$ which is not contained in $\B_-(L)$. Then $V$ is uniruled if either $-K_X$ or $L-K_X$ is nef.
\end{prop}
\begin{proof} By Lemma \ref{lem:decomp} there exists a rational number $\alpha>0$ and a decomposition 
$$\alpha L\equiv A+E$$
with $A$ an ample $\Q$-divisor and $E$ an effective $\Q$-divisor such that $V$ is a maximal lc center for $(X,E)$.

Suppose first that $V$ is a component of $\B(L)$. In case $-K_X$ is nef we write
$$D:=mL-K_X=\left((m-\a)L+\frac 1 2 A\right)+\left(\frac 1 2 A-K_X\right)+E$$
and the result follows from Theorem~\ref{thm:takauniruled}, item (a), applied to $D$ since $\frac 1 2 A-K_X$ is ample and $V$ is not contained in $\B((m-\a)L+\frac 1 2 A)$ for $m\gg 1$ since it is not contained in $\B_-(L)$ by assumption. In case $L-K_X$ is nef we write
$$D:=(m+1)L-K_X+=\left((m-\a)L+\frac 1 2 A\right)+(\frac 1 2 A+L-K_X)+E$$
and we conclude similarly since $\frac 1 2 A +L-K_X$ is ample and $V$ is not contained in $\B((m-\a)L+\frac 1 2 A)$ for $m\gg 1$. 

Assume now that $V$ is a component of $\B_+(L)$ not contained in $\B(L)$. In case $-K_X$ is nef we write
$$D:=mL-K_X=(m-\a)L+(A-K_X)+E$$
and in case $L-K_X$
$$D:=(m+1)L-K_X=(m-\a)L+(A+L-K_X)+E$$
and conclude as above by applying Theorem~\ref{thm:takauniruled}, item (b), to $D$.
\end{proof}
Proposition~\ref{prop1} already proves $(ii)$ of Theorem B in case the component $V$ of either $\B(L)$ or $\B_+(L)$ is not contained in $\B_-(L)$. We now focus on the case where $V$ is a component of $\B_-(L)$. It is then as before a component of $\B(L+\e A)$ if $A$ is ample and $\e>0$ is small enough, but this does not directly reduce case $(i)$ to case $(ii)$ since $L+\e A$ is not a line bundle anymore, and one thus has to exercise a little more care in the reduction trick.  We will argue as in \cite[Proof of Proposition 6.1, (2)]{taka2}. 

Let thus
$$t_0:= \inf\{t\in\Q, V\subset\B(tL+A)\}.$$
Note that $t_0>0$ since $tL+A$ is ample for $0<t\ll 1$.  On the other hand we also have $t_0<\infty,$ by \cite[Lemma 2.5, item (1)]{taka2}.

\begin{lem}\label{claim}
There exist two positive integers $m$ and $n$ such that
\begin{eqnarray}\label{eq:xy0}
\frac{m+1}{n}>t_0
\end{eqnarray}
and such that
\begin{eqnarray}\label{eq:xy1}
mL+nA\sim_\Q A_1+E_1
\end{eqnarray}
and
\begin{eqnarray}\label{eq:xy2}
(m+1)L+nA\sim_\Q A_2+E_2
\end{eqnarray}
where, for each $i=1,2$, the divisor $A_i$ is an ample $\Q$-divisor and $E_i$ is an effective $\Q$-divisor such that $V$ is a maximal lc center for $(X,E_i)$.
\end{lem}
Assume this result for the moment. By the definition of $t_0$,  condition (\ref{eq:xy0}) guarantees that $V\subset\B((m+1)L+nA)$. If $-K_X$ is nef, we write
$$D:=(m+1)L+nA-K_X\buildrel{(\ref{eq:xy1})\ \ }\over{\sim_\Q}(A_2-K_X)+E_2$$
and we conclude by Theorem~\ref{thm:takauniruled}, item (i), since $A_2-K_X$ is ample. If $L-K_X$ is nef we write
$$D:=(m+1)L+nA-K_X\buildrel{(\ref{eq:xy2})\ \ }\over{\sim_\Q}(A_1+L-K_X)+E_1$$
and we conclude as before since $A_1+L-K_X$ is nef.

\begin{proof}[Proof of Lemma~\ref{claim}] Choose an integer $m_1>t_0+1$. Since $V$ is an irreducible component of $\B(m_1L+A)$ we can apply Lemma~\ref{lem:decomp} and write
\begin{equation}\label{equ:D}\a(m_1L+A)=H+F
\end{equation}
where $\alpha$ is a positive rational number, $H$ is ample and $V$ is a maximal lc center for $(X,F)$. Now choose two positive integers $m$ and $n$ such that 
\begin{eqnarray}\label{eq:xy3}
m>\max\{\alpha m_1, t_0\},
\end{eqnarray}
\begin{eqnarray}\label{eq:xy4}
n> \max\{\alpha , 1\}
\end{eqnarray}
and
\begin{eqnarray}\label{eq:xy5}
-1<m-nt_0<\a.
\end{eqnarray}
The existence of such $m$ and $n$ may be seen as follows: if $t_0\in\Q$, then take $m$ and $n$ two sufficiently divisible integers such that $t_0=m/n$. If $t_0\not\in\Q$, then the existence follows from elementary diophantine approximation.

Since $m_1>t_0+1$, by conditions (\ref{eq:xy3}) and (\ref{eq:xy5}), {\it for all} such integers  $m,n$ we have that 
\begin{eqnarray}\label{eq:<t_0}
\frac{m-\a m_1}{n-\a}<t_0.
\end{eqnarray}
Now notice that, for every $\epsilon >0$, among the integers satisfying (\ref{eq:xy3}) and (\ref{eq:xy5}) we can choose $n$ big enough such that 
$$
 \frac{1}{n-\a}<\epsilon.
$$
In conclusion there exist $m$ and $n$ satisfying (\ref{eq:xy3}), (\ref{eq:xy4}) and (\ref{eq:xy5}) and such that we also have
\begin{eqnarray}\label{eq:<t_0bis}
\frac{m-\a m_1}{n-\a}<\frac{(m+1)-\a m_1}{y-\a}= \frac{m-\a m_1}{n-\a}+\frac{1}{n-\a}<t_0.
\end{eqnarray}
From (\ref{eq:<t_0bis}) and the definition of $t_0$ one deduces (as in \cite[Lemma 2.5, item (2)]{taka2}) the existence of the following decompositions
\begin{eqnarray}\label{eq:dec1}
(m-\a m_1)L+(n-\a)A\sim_\Q A'_1+ E'_1
\end{eqnarray}
and 
\begin{eqnarray}\label{eq:dec2}
(m+1-\a m_1)L+(n-\a)H\sim_\Q A'_2+ E'_2
\end{eqnarray}
where, for each $i=1,2$, the divisor $A_i$ is an ample $\Q$-divisor and $E'_i$ is an effective $\Q$-divisor such that 
\begin{eqnarray}\label{eq:star}
V\not\subset \supp(E'_i).
\end{eqnarray}

To conclude the proof of the Lemma, notice that thanks to (\ref{eq:dec1}) and (\ref{eq:dec2}), we can write 
\begin{eqnarray*}
mL+nA=\a(m_1L+A)+ \big( (m-\a m_1)L +(n-\a)A \big)
\equiv A'_1+H + (F+E'_1)
\end{eqnarray*}
and 
\begin{eqnarray*}
(m+1)L+nA=\a(m_1L+A)+ \big( (m+1-\a m_1)L +(n-\a)A \big)
\equiv A'_2+H + (F+E'_2).
\end{eqnarray*}
The proof is now concluded by setting $A_i:=A_i'+H$ and $E_i:=F+E_i'$ for $i=1,2$. 

\end{proof}

\vskip 30pt

\noindent
{\small S\'ebastien Boucksom\\
CNRS--Universit\'e Paris 7\\
Institut de Math\'ematiques\\
F-75251 Paris Cedex, France\\
E-mail : {\tt boucksom@math.jussieu.fr}}
\vskip 1em

\noindent
{\small Ama\"el Broustet\\
Universit\'e Lille 1\\
UMR CNRS 8524\\
UFR de math\'ematiques\\
59 655 Villeneuve d'Ascq Cedex, France\\
E-mail : {\tt broustet@math.univ-lille1.fr} }

\vskip 1em

\noindent
{\small  Gianluca Pacienza\\
Institut de Recherche Math\'ematique Avanc\'ee\\
Universit\'e de Strasbourg et CNRS\\ 
7, Rue R. Descartes - 67084 Strasbourg Cedex, France \\
E-mail : {\tt pacienza@math.unistra.fr}}

\end{document}